\documentclass[12pt]{article}		

\usepackage{amsmath}\usepackage{amssymb}\usepackage{amsfonts}\usepackage{amsthm} 
\newtheorem{theorem}{Theorem}
\newtheorem{proposition}{Proposition}
\newtheorem{lemma}[theorem]{Lemma}

\theoremstyle{definition}

\theoremstyle{definition}
\newtheorem{remark}{Remark}
\theoremstyle{definition}
\newtheorem{example}{Example}

\numberwithin{equation}{section}
\numberwithin{table}{section}
\numberwithin{figure}{section}

\newcommand{\Z}{\mathbb{Z}} 
\newcommand{\R}{\mathbb{R}}
\newcommand{\N}{\mathbb{N}}
\newcommand{\lpk}{l^{p_k}(\Z)} 
\newcommand{\e}{\varepsilon} 

\newcommand{\abs}[1]{\left|#1\right|}
\newcommand{\norm}[1]{\left\|#1\right\|}

\DeclareMathOperator{\argmin}{argmin}

\title{Sequences of positive homoclinic solutions to difference equations with variable exponent}

\author{Robert Stegli\'{n}ski and Magdalena Nockowska-Rosiak\footnote{Corresponding author}\\ \\
Institute of Mathematics, Lodz University of Technology,\\
ul. W\a'olcza\a'nska 215, 90-924, \L\a'od\a'z, Poland\\
emails: robert.steglinski@p.lodz.pl, magdalena.nockowska@p.lodz.pl
\date{} 
}
\begin{document}\maketitle

{\bf Abstract.} We study the existence of infinitely many positive homoclinic solutions to a second-order difference equation on integers with $p_k$-Laplacian. % and unbounded potential. 
To achieve our goal we use the critical point theory and the general variational principle of Ricceri.

%Basing on the critical point theory, we prove the existence of infinitely many positive homoclinic solutions to a discrete problem 
%$$
%-\nabla^- \left( a_k\abs{\nabla^+ u_k}^{p_k-2}\nabla^+ u_k\right) +b_k\abs{u_k}^{p_k-2}u_k=f_k(u_k), \quad k\in\Z. 
%$$
%We consider the nonlinear term $f_k:\R\to\R$, $k\in\Z$ with an appropriate behavior at infinity or zero.
%%,without any symmetry assumptions. %The approach is based on the critical point theory.

% Please enter at most 6 keywords here with lowercase letters separated by commas.

% Please enter at most 5 Mathematics Subject Classification codes here. Please use 2010 classification codes, which can be found on the following link: http://www.ams.org/msc//msc2010.html.
%\noindent{\bf{2010 Mathematics Subject Classification:}} 39A10, 47J30, 35B38
%}

\section{Introduction}

In this paper we consider the nonlinear second-order difference problem of the form
\begin{equation}\label{eq}
\begin{cases}
-\nabla^- \left( a_k\abs{\nabla^+ u_k}^{p_k-2}\nabla^+ u_k\right) +b_k\abs{u_k}^{p_k-2}u_k=f_k(u_k),& k\in\Z 
\\
u_k\rightarrow 0   &\textrm{as} \ |k|\to \infty
\end{cases}.
\end{equation}
%\begin{align}
%\left\{
%\begin{array}{ll}
%-\nabla^- \left( a_k(\nabla^+ u_k)^{p_k-2}\nabla^+ u_k\right) +b_k\abs{u_k}^{p_k-1}u_k=f_k(u_k), 
 %& \mbox{for all $k\in\mathbb{Z}$} \\ u(k)\rightarrow 0 & \mbox{as $|k|\to \infty$}.%
%\end{array}%
%\right.  \label{eq}
%\end{align}
In the whole paper we assume $\{p_k\}_{k\in\Z}$ with $1<p^-\le p^+<\infty$, $p^-:=\inf_{k\in\Z}p_k$, $p^+:=\sup_{k\in\Z}p_k$, %$\phi_{p}(t)=|t|^{p-2}t$ for all $t\in {\mathbb{R}}$, 
$a,b:\Z\to (0,+\infty)$, and $f_k:\R\to\R$ are continuous functions for $k\in\Z$. Moreover, we use the notation $\nabla^- u_k=u_k-u_{k-1}$, $\nabla^+ u_k=u_{k+1}-u_{k}$, $k\in\Z$. We say that a solution $u=\{u_k\}$ to equation given in \eqref{eq} is homoclinic if $\lim_{\abs{k} \to \infty }u_k=0.$

Difference equations represent the discrete counterpart of ordinary differential equations and are usually studied in connection with numerical analysis. Equation given in \eqref{eq} is the discrete counterpart of the following
\begin{equation}\label{eq-c}
-\left(a(x)\abs{u'}^{p(x)-2}u'\right)'+b(x)\abs{u}^{p(x)-2}u=f(x,u).
\end{equation}
The above equation is interesting because of its wide applications for example in nonlinear elasticity or fluid dynamics. Moreover, equation \eqref{eq-c} becomes the stationary nonlinear Schr\"odinger equation for $p(x)\equiv2$ and $a(x)\equiv1$. As its continuous counterpart, equation given in \eqref{eq} for $p_k=2$ has many applications in various areas of physics.

The variational method is very powerful tool in study of boundary value problems to difference equations. Many authors have applied different results of the critical point theory to prove existence and multiplicity results for solutions to
discrete nonlinear problems. Studying such problems on bounded discrete
intervals allows to search for solutions in a finite-dimensional Banach
space (see \cite{CIT, CM, BR1, RR}). To study such problems on unbounded intervals directly by
variational methods, \cite{IT} and \cite{MG} introduced coercive weight
functions which allow to preserve of certain compactness properties on $
l^{p}$-type spaces. That method was used in the following papers \cite{IR,K,St2016ADE,St2016,St2017, St2018, SM}.

\bigskip

The goal of the present paper is to establish sufficient conditions for the existence of a sequence of positive solutions to problem (\ref{eq}) which norm tend to infinity or zero. Infinitely many solutions for a constant $\{p_k\}$ were obtained in \cite{SM}  by employing Nehari manifold methods, in \cite{K, St2018} by applying a variant of the fountain theorem, in \cite{St2016ADE,St2016,St2017} by use of the Ricceri's theorem (see \cite{BMB, R}) and in \cite{St2016, St2017} by applying a direct argumentation. The existence of a nontrivial or infinitely many solutions to problem \eqref{eq} with variable exponent were proved in \cite{AP, ChTA, KA} by the mountain pass theorem. In \cite{AP, ChTA} authors assumed some symmetry conditions on nonlinear part of the problem to get infinitely many solutions to it. Moreover, in \cite{AP} there were considered different sequences of variable exponent on the left side of equation given in \eqref{eq}. In this paper, similarly to \cite{St2016ADE,St2016, St2017}, the nonlinearity $f$ has a suitable behavior at infinity or zero, without any symmetry conditions.

We will study problem \eqref{eq} under the following assumptions:
\begin{itemize}
%	\item[$(A)$] $a_k>0$, $b_k\geq b_{0}>0$ for all $k\in\Z $, $b_k\rightarrow +\infty $ as $\abs{k}\to+\infty$;
	
	\item[$(A)$] $a_k>0$, $b_k>0$ for all $k\in\Z $, $b_k\rightarrow +\infty $ as $\abs{k}\to+\infty$;
	
	\item[$(F_1)$] $f_k(0)=0$ for all $k\in\Z$;
	
	\item[$(F_2)$] $\displaystyle\sum_{k\in\Z}\max_{\abs{t}\le T}\abs{f_{k}(t)}<\infty$ for any $T>0$; %$\displaystyle\sup_{\abs{t}\le T}\abs{f_{\cdot}(t)} \in l_{1}$ for any $T>0$;

	%\item[$(F_2')$] $\displaystyle\lim_{t\to 0^+}\frac{f_k(t)}{ t^{p^+-1}}=0$ uniformly for all $k\in \Z$;
	
	\item[$(F_3^0)$] there exist sequences $\{c_{n}\},\{d_{n}\}$ such that $0<d_{n+1}<c_{n}<d_{n},$ $\lim_{n\rightarrow \infty }d_{n}=0$ and $f_k(t)\leq 0$ for every $k\in\Z$ and $t\in [ c_{n},d_{n}],n\in \N$;
	
	\item[$(F_3^\infty)$] there exist sequences $\{c_{n}\},\{d_{n}\}$ such that $0<c_{n}<d_{n}<c_{n+1},$ $\lim_{n\rightarrow \infty }c_{n}=+\infty$ and $f_k(t)\leq 0$ for every $k\in\Z$ and $t\in [ c_{n},d_{n}],n\in \N$;

	\item[$(F_4^0)$] $\displaystyle\liminf_{t\to0^+}\frac{\sum_{k\in\Z}\max_{\abs{\xi}\le t}F_k(\xi)}{t^{p^+}}<\frac{1}{p^+\alpha^{p^+}}$;
	
	\item[$(F_4^\infty)$] $\displaystyle\liminf\limits_{t\to+\infty}\frac{\sum_{k\in\Z}\max_{\abs{\xi}\le t}F_k(\xi)}{t^{p^-}}<\frac{1}{p^+\alpha^{p^-}}$;

	\item[$(F_5^0)$] %$\displaystyle\limsup\limits_{(k,t)\rightarrow (+\infty,0^+)}\frac{F_k(t)}{\left( a_{k-1}+a_k+b_k\right) t^{p^-}}>\tfrac{1}{p^-}\quad\vee\quad\limsup\limits_{(k,t)\rightarrow (-\infty,0^+)}\frac{F_k(t)}{\left( a_{k-1}+a_k+b_k\right) t^{p^-}}>\tfrac{1}{p^-}$;
		%\\
		$\displaystyle\max\left\{\limsup\limits_{(k,t)\rightarrow (+\infty
		,0^+)}\frac{F_k(t)}{\left( a_{k-1}+a_k+b_k\right) t^{p^-}},\limsup\limits_{(k,t)\rightarrow (-\infty,0^+)}\frac{F_k(t)}{\left( a_{k-1}+a_k+b_k\right) t^{p^-}}\right\}>\tfrac{1}{p^-}$;
		
	\item[$(F_5^{\infty})$] %$\displaystyle\limsup\limits_{(k,t)\rightarrow (+\infty,+\infty)}\frac{F_k(t)}{\left( a_{k-1}+a_k+b_k\right) t^{p^+}}>\tfrac{1}{p^-}\ \vee\ %$(F_5^{-\infty})$ \limsup\limits_{(k,t)\rightarrow (-\infty,+\infty)}\frac{F_k(t)}{\left( a_{k-1}+a_k+b_k\right) t^{p^+}}>\tfrac{1}{p^-}$;
	$\displaystyle\max\left\{\limsup\limits_{(k,t)\rightarrow (+\infty
		,+\infty)}\frac{F_k(t)}{\left( a_{k-1}+a_k+b_k\right) t^{p^+}},
	\limsup\limits_{(k,t)\rightarrow (-\infty,+\infty)}\frac{F_k(t)}{\left( a_{k-1}+a_k+b_k\right) t^{p^+}}\right\}>\tfrac{1}{p^-}$;

	\item[$(F_6^0)$] $\displaystyle\sup_{k\in \Z}\left( \limsup\limits_{t\to0^+ }\frac{F_k(t)}{\left(
	a_{k-1}+a_k+b_k\right) t^{p^-}}\right)>\tfrac{1}{p^-}$;

	\item[$(F_6^\infty)$] $\displaystyle\sup_{k\in \Z}\left( \limsup\limits_{t\to +\infty }\frac{F_k(t)}{\left(
	a_{k-1}+a_{k}+b_k\right) t^{p^+}}\right)>\tfrac{1}{p^-}$, % \vee \sup_{k\in \Z}\left( \limsup\limits_{t\to -\infty }\frac{F_k(t)}{\left(a_{k-1}+a_{k}+b_k\right) t^{p^+}}\right)>\tfrac{1}{p^-}$,%\\  $\displaystyle\sup_{k\in \Z}\left( \limsup\limits_{t\to \pm\infty }\frac{F_k(t)}{\left(a_{k-1}+a_{k}+b_k\right) t^{p^+}}\right)>\tfrac{1}{p^-}$
\end{itemize}
where $F_k$ is the primitive function of $f_k$, that is $F_k(t)=\int_{0}^{t}f_k(s) ds$ for $t\in \R$, $k\in \Z$, and
\begin{equation}\label{alfa}
\alpha:=\sup_{k\in\Z}b^{-\tfrac{1}{p_k}}_k.
\end{equation}
Notice that $\alpha<\infty$, because $b_k\to\infty$ and $\{p_k\}$ is bounded.
%\begin{remark}
%It is easy to see that condition $(A)$ is equivalent to
%\begin{itemize}
	%\item[$(A')$] $a_k>0$, $b_k\geq b_{0}>0$ for all $k\in\Z $, $b_k\rightarrow +\infty $ as $\abs{k}\to+\infty$.
%\end{itemize}	
%\end{remark}
%\begin{remark}
%Examples of functions $f_k$ which satisfy the previously mentioned assumptions are slight modification of examples from \cite{St2016, St2017}, for any sequence $\{p_k\}$ with $1<p^-<p^+<\infty$. This a consequence of the fact that limits in examples for $(F^0_5)$, $(F^\infty_5)$, $(F^0_6)$, $(F^\infty_6)$ in \cite{St2016, St2017} are equal to 0 or $+\infty$. The only difference in assumptions concerns $(F_2)$ which is independent from $\{p_k\}$. In all examples for any $T>0$ the set
%$$
%\{k\in\Z:\max_{\abs{t}\le T}\abs{f_k(t)}>0\}
%$$ 
%is finite, so $(F_2)$ is satisfied.
%\end{remark}

Equation given in \eqref{eq} is considered in certain reflexive Banach space $E$ defined in the next Section. This equation is the Euler-Lagrange equation of certain action functional which is well defined on space $E$. Elements of this space automatically satisfy boundary condition in  problem \eqref{eq}. In this paper we will prove the following theorems:
\begin{theorem}\label{tw1} 
Assume that $(A)$, $(F_1)$, $(F_{2})$ and $(F_3^\infty)$ are satisfied. Moreover, assume that one of conditions $(F_5^\infty)$ or $(F_6^\infty)$ is satisfied. Then, problem \eqref{eq} possesses a sequence of positive solutions in $E$ whose norms tend to infinity.
\end{theorem}
\begin{theorem}\label{tw2}
Assume that $(A)$, $(F_1)$, $(F_{2})$ and $(F_3^0)$ are satisfied. Moreover, assume that one of conditions $(F_5^0)$ or $(F_6^0)$ is satisfied. Then, problem \eqref{eq} has a sequence of positive solutions in $E$ whose norms tend to 0.
\end{theorem}
\begin{theorem}\label{tw2'} 
Assume that $(A)$, $(F_{1})$, $(F_2)$ and $(F_{4}^0)$ are satisfied and assume that one of the conditions $(F_5^0)$ or $(F_6^0)$ holds. Then, problem \eqref{eq} possesses a sequence of positive solutions in $E$ whose norms tend to zero.
\end{theorem}
\begin{theorem}\label{tw1'} 
Assume that $(A)$, $(F_{1})$, $(F_2)$ and $(F_{4}^\infty)$ are satisfied and
assume that one of the conditions $(F_5^\infty)$ or $(F_6^\infty)$ holds. Then, problem \eqref{eq} has a sequence of positive solutions in $E$ whose norms tend to infinity.
\end{theorem}
To get our goal in Theorems \ref{tw1}, \ref{tw2} a direct variational approach is used, introduced in \cite{FK} and then used in such papers as 
\cite{KMR, KMT, KRV, Ra, St2017}. In the proof for Theorems \ref{tw2'}, \ref{tw1'} the
general variational principle of Ricceri is used, which was applied in \cite{BCh,BMBR,St2016ADE,St2017}. 
%To obtain the differentiability of the energy functional associated with problem \eqref{eq}, so far in the literature the following condition has been used 
%\begin{equation*}
%\lim_{t\rightarrow 0}\frac{\left\vert f(k,t)\right\vert }{\left\vert t\right\vert ^{p-1}}=0\text{ \ uniformly for all }k\in \Z}
%\end{equation*}%
%following \cite{IT,MG} and then used in \cite{St,St1,SM}.

%\bigskip We cannot use the condition, as it contradicts each of $(F_{5}^{+}), $ $(F_{5}^{-}),$ $(F_{6})$ conditions

%We obtain our results due to a suitable oscillatory behavior of the nonlinearity $f$. Let us observe that to satisfy the condition $(F_{6})$ it
%suffices that a suitable oscillatory behavior is present for just one $k\in \Z$, while for satisfying conditions $(F_{5}^{+})$ or $(F_{5}^{-})$ a suitable behavior of the nonlinearity $f$ needs to be maintained for an infinite number of $k\in\Z$.

The plan of the paper is as follows: Section 2 is devoted to the abstract framework, in Section 3 proofs of Theorems \ref{tw1}-\ref{tw1'} and examples are presented.

\section{Abstract framework}
Recall that the sequence $\{p_k\}_{k\in\Z}$ is bounded with $1< p^-\le p^+<\infty$. We introduce 
$$
\lpk=\left\{ u:\Z\to\R :\
\sum_{k\in \Z}\abs{u_k}^{p_k} <\infty \right\} 
$$
with norm 
$$
\norm{u}_{p_k} =\inf\left\{\eta>0: \sum_{k\in \Z}\abs{\tfrac{ u_k}{\eta}}^{p_k}\le1\right\},
$$
and %The solutions are found in the space $(E,\norm{\cdot}_E )$, where 
$$
E=\left\{ u\in\lpk :\
\sum_{k\in \Z}\left( a_k\abs{\nabla^+ u_k}^{p_k}+b_k\abs{u_k}^{p_k}\right) <\infty \right\} 
$$
with norm
$$
\norm{u}_E =\inf\left\{\eta>0: \sum_{k\in \Z}\left( a_k\abs{\tfrac{\nabla^+ u_k}{\eta}}^{p_k}+b_k\abs{\tfrac{u_k}{\eta}}^{p_k}\right)\le1\right\}.
$$
\begin{proposition}
For $u\in E$ we have
\begin{equation}\label{alfa-osz}
\norm{u}_{\infty}\le \alpha\norm{u}_E,
\end{equation}
where $\norm{u}_{\infty}=\sup\{\abs{u_k}:k\in\Z\}$ and $\alpha$ is defined in \eqref{alfa}.
\end{proposition}
\begin{proof}
Let $u\in E$ and $\varepsilon>0$. Choose $\eta>0$ such that $\norm{u}_E\le\eta\le\norm{u}_E+\varepsilon$. Then $b_k\abs{\tfrac{u_k}{\eta}}^{p_k}\le1$, and so $\abs{u_k}\le\alpha\eta$, for any $k\in\Z$. This implies that $\norm{u}_\infty\le\alpha(\norm{u}_E+\varepsilon)$. Passing to $\varepsilon\to0$ we get $\norm{u}_\infty\le\alpha\norm{u}_E$.
\end{proof}

\begin{theorem}\cite{KA}
$(\lpk,\norm{\cdot}_{p_k} )$ and $(E,\norm{\cdot}_E )$ are a reflexive Banach space and embeddings $E\hookrightarrow \lpk$ and  $E\hookrightarrow  l^{p^+}(\Z)$ are compact.
\end{theorem}
Let
$$
\rho_{E}(u):=\sum_{k\in \Z}\left( a_k\abs{\nabla^+ u_k}^{p_k}+b_k\abs{u_k}^{p_k}\right), \quad u\in E.
$$
It is easy to see that
\begin{proposition}\label{osz-Orl}\cite{KA}
For $u\in E$ we have the following relations:
\begin{itemize}
\item[(i)] $\norm{u}_{E}<1$ $(=1,$ $>1)$  $\Leftrightarrow$ $\rho_{E}(u)<1$ $(=1,$ $>1)$;
\item[(ii)] $\norm{u}_{E}\le1$   $\Rightarrow$ $\norm{u}^{p^+}_E\le\rho_{E}(u)\le\norm{u}^{p^-}_E$;
\item[(iii)] $\norm{u}_{E}\ge1$   $\Rightarrow$ $\norm{u}^{p^-}_E\le\rho_{E}(u)\le\norm{u}^{p^+}_E$;
\item[(iv)] $\rho_E(u^n)\to0$, as $n\to\infty$   $\Leftrightarrow$ $\norm{u^n}_E\to0$, as $n\to\infty$.
\end{itemize}
\end{proposition}

Now, we define a functional $\Phi$ on $E$ by formula 
		\begin{equation}\label{Phi}
	\Phi (u):=\sum_{k\in \Z}\left( \tfrac{a_k}{p_k}\abs{\nabla^+u_k}^{p_k}+\tfrac{b_k}{p_k}\abs{u_k}^{p_k}\right),\ u\in E.% \ \ \ \text{for all \ \ \ }u\in E
	\end{equation}
Note that  $p^->1$ implies that $\Phi:E\to\R$ is well defined. By the definition of $\Phi$ and Proposition \ref{osz-Orl} we get 
	\begin{proposition}\label{osz-Phi}
	For $u\in E$ we have the following inequalities:
	\begin{itemize}
	\item[(i)] $\tfrac{1}{p^{+}}\rho_{E}(u)\le \Phi(u)\le \tfrac{1}{p^{-}}\rho_{E}(u)$;
	\item[(ii)] $\norm{u}_E\le 1\ \Rightarrow \ \norm{u}_{\infty}\le\alpha\norm{u}_E\le\alpha\left(p^+\Phi(u)\right)^{\tfrac{1}{p^+}}$;
	\item[(ii)] $\norm{u}_E\ge 1\ \Rightarrow \ \norm{u}_{\infty}\le\alpha\norm{u}_E\le\alpha\left(p^+\Phi(u)\right)^{\tfrac{1}{p^-}}$.
	\end{itemize}
	\end{proposition}
	\begin{lemma}\label{phi-C1}
	Assume that $(A)$ is satisfied. Then $\Phi\in C^1(E,\R)$ with
	\begin{equation}\label{poch-phi}
	\left\langle \Phi'(u),v \right\rangle=\sum_{k\in\Z}\left(a_k\abs{\nabla^+u_k}^{p_k-2}\nabla^+u_k\nabla^+v_k+b_k\abs{u_k}^{p_k-2}u_kv_k\right),
	\end{equation}
	for all $u,v\in E$. Moreover,  $\Phi$ is a sequentially weakly lower semicontinuous functional on $E$.
	\end{lemma}
	\begin{proof} The standard proof is omitted, see \cite{ChTA}, \cite{IT}, \cite{St2017}.
\end{proof}
The following notation will be used in the rest of the paper:
$$
t^{+}=\max\{t,0\},\quad t^-=\min\{0,t\},\quad f^+_k(t)=f_k(t^+),\quad F^+_k(t)=\int^t_{0}f^+_k(s)\,ds,
$$
for $t\in\R$, $k\in\Z$. 
\begin{remark}
If $(F_1)$ is satisfied, then $f^+_k$ is a continuous function on $\R$ for $k\in\Z$. Moreover, notice that $F^+_k(t)=F_k(t)$ for $t\ge0$, $k\in\Z$. This implies that if one of conditions $(F^0_4)$, $(F^\infty_4)$, $(F^0_5)$, $(F^\infty_5)$, $(F^0_6)$, $(F^\infty_6)$ is satisfied for $F_k$, then this conditions is satisfied for $F^+_k$, too.
\end{remark}

We define a functional $\Psi$ on $\lpk$ by formula
\begin{equation}\label{Psi}
	\Psi (u):=\sum_{k\in \Z}F^+_k(u_k), \ u\in \lpk. %\text{ \ \ for all \ \ }u\in \lpk 
	\end{equation}
%	where $F_k(s)=\int_{0}^{s}f_k(t)dt$ for $s\in \R$ and $k\in\Z$. 
	\begin{lemma}\label{psi-C1}
	Assume that $(F_1)$, $(F_2)$ hold. Then $\Psi\in C^1(\lpk,\R)$ with
	\begin{equation}\label{poch-psi}
	\left\langle \Psi'(u),v \right\rangle=\sum_{k\in\Z}f^+_k(u_k)v_k,
	\end{equation}
	for all $u,v\in \lpk$. Moreover, $\Psi$ is a sequentially weakly  continuous functional on $E$. %$\lpk$.
	\end{lemma}
	\begin{proof}
	In \cite{St2017} there was shown that if $(F_2)$ is satisfied, then for any $p>1$ functional $\Psi:l^p\to\R$ given by formula $\Psi(u)=\sum_{k\in\Z}F^+_{k}(u_k)$, $u=\{u_k\}\in l^p$, is continuously differentiable on $l^p$. This and continuous emending $\lpk\hookrightarrow l^{p^+}$ gives that $\Psi\in C^1(\lpk,\R)$.
	
	Now we prove that $\Psi$ is a sequentially weakly continuous functional on $E$. Let $u^n\rightharpoonup u$ in $E$. Set $T:=\max\{\norm{u}_{\infty},\sup_{n\in\N}{\norm{u^n}_{\infty}}\}$. By mean value theorem, there exists $w^n_k$ for any $n\in\N$, $k\in\Z$ which belongs to interval with ends $u^n_k$, $u_k$ such that $F^+_k(u^n_k)-F^+_k(u_k)=f^+_k(w^n_k)(u^n_k-u_k)$. Hence
	\begin{align*}
	&\abs{\Psi(u^n)-\Psi(u)}\le\sum_{k\in\Z}\abs{F^+_k(u^n_k)-F^+_k(u_k)}=\sum_{k\in\Z}\abs{f^+_k(w^n_k)}\abs{u^n_k-u_k}
	\\
	&\le\norm{u^n-u}_{\infty}\sum_{k\in\Z}\max_{\abs{t}\le T}{\abs{f^+_k(t)}}\le\norm{u^n-u}_{\infty}\sum_{k\in\Z}\max_{\abs{t}\le T}\abs{f_k(t)}.
	\end{align*} 
	From the above estimation and the fact that emending $E\hookrightarrow l^\infty$ is compact, we have that for any subsequence $\{u^{m_n}\}$ of $\{u^n\}$ there exists $\{u^{l_{m_n}}\}\subset\{u^{m_n}\}$ such that $\norm{u^{l_{m_n}}-u}_{\infty}\to0$ and $\Psi(u^{l_{m_n}})\to\Psi(u)$, as $n\to\infty$. Hence $\Psi(u^n)\to\Psi(u)$ as $n\to\infty$.
	%Analogously to first part of Lemma 2.1 in [EJDE 2017 Steglinski] we get that
	%$$
	%\left\langle \Psi'(u),v \right\rangle=\sum_{k\in\Z}f_k(u_k)v_k,
	%$$
	%for all $u,v\in \lpk$. To prove that $\Psi$ is Gateaux differentiable notice that $(F_2)$ implies that $\sum_{k\in\Z}\abs{f_k(u_k)}<\infty$ for any $u\in\lpk$. Hence there exists $h\in\N$ such that $\abs{f_k(u_k)}<1$ for $\abs{k}>h$. Therefore
	%$$
	%\sum_{\abs{k}>h}\abs{f_k(u_k)}^{q_k}<\sum_{\abs{k}>h}\abs{f_k(u_k)}
	%$$
	%where $\displaystyle\tfrac{1}{q_k}+\tfrac{1}{p_k}=1\displaystyle$. This means that $(f_{\cdot}(u_{\cdot}))\in l^{q_k}$. From Kuang lemma 11 we get that
	%$$
	%\abs{\left\langle \Psi'(u),v \right\rangle}\le\left(\tfrac{1}{p^-}+\tfrac{1}{q^-}\right)\norm{f_{\cdot}(u_{\cdot})}_{q_k}\norm{v}_{p_k},
	%$$
	%where $q^-=\inf\{q_k:k\in\Z\}$. The proof of continuity of $\Psi'$ is similar to third part of Lemma 2.1 in [EJDE 2017 Steglinski].
	\end{proof}
	%\begin{lemma}\label{J}
	%Assume that $(F_2)$ holds. Then $\Phi$, is sequentially weakly lower semicontinuous functional on $E$, and $\Psi$ is sequentially weakly  continuous functional on $E$.
	%\end{lemma}
	%\begin{proof}
		%Now we prove that $\Psi$ is sequentially weakly continuous functional on $E$. Let $u^n\rightharpoonup u$ in $E$. Since emending $E\hookrightarrow l^\infty$ is compact, then $\norm{u^n-u}_{\infty}\to0$, as $n\to\infty$, choosing a subsequence, if necessary. Let $T:=\max\{\norm{u}_{\infty},\sup_{n\in\N}{\norm{u^n}_{\infty}}\}$. By mean value theorem for any $n\in\N$, $k\in\Z$ there exists $w^n_k$ which belongs to interval with ends??? $u^n_k$, $u_k$ such that 
	%$$
	%\abs{\Psi(u^n)-\Psi(u)}\le\sum_{k\in\Z}\abs{F_k(u^n_k)-F_k(u_k)}=\sum_{k\in\Z}\abs{f_k(w^n_k)}\abs{u^n_k-u_k}\le \sup_{\abs{t}\le T}\norm{f_\cdot(t)}_{l^1}\norm{u^n-u}_{\infty}.%\to0
	%$$
	%\end{proof}
	Let $J:E\rightarrow \R$ be a functional associated to problem 
	\eqref{eq} defined by 
	\begin{equation}\label{defJ}%\begin{align*}
	J(u)=\Phi (u)- \Psi (u),
	\end{equation}%\end{align*}
	where $\Phi$ is given by \eqref{Phi} and $\Psi$ is given by	\eqref{Psi}. 
	\begin{proposition}\label{pom}
	Assume that $(A)$, $(F_1)$ and $(F_{2})$ are satisfied, then $J\in C^{1}(E,\R)$ and $J$ is a sequentially weakly lower semicontinuous functional on $E$. Moreover, every critical point $u\in E\setminus\{0\}$ of $J$ is a positive homoclinic solution to problem \eqref{eq}.
	\end{proposition}
	\begin{proof}
	By Lemma \ref{phi-C1} and Lemma \ref{psi-C1} we get that $J\in C^{1}(E,\R)$ and $J$ is a sequentially weakly lower semicontinuous functional on $E$. It is easy, to see that, for $u$ a critical point of $J$, we have
	$$
	f^+_k(u_k)u^-_k=0, \quad 0=\left\langle J'(u),-u^- \right\rangle=-\left\langle \Phi'(u),u^- \right\rangle\le-\rho_{E}(u^-).
	$$
	Hence $u^-_k=0$, for $k\in\Z$ and $u$ is nonnegative. It means that $u$ is a homoclinic solution to \eqref{eq}. 
	%Putting $v=e^i$, where $e^i_k=1$ if $i=k$ and $e^i_k=0$ otherwise, for $i,k\in\Z$, into \eqref{poch-phi}, \eqref{poch-psi} we get that $u$ is a nonnegative solution to \eqref{eq}. 
	%Now we assume, the $u$ is a non-trivial critical point of $J$. 
	If there exists $i\in\Z$ such that $u_i=0$, then by $(F_1)$ we get that
	$$
	a_i\abs{u_{i+1}}^{p_i-2}u_{i+1}+a_{i-1}\abs{u_{i-1}}^{p_i-2}u_{i-1}=0.
	$$
	By non-negativity of each term of the above equality and $(A)$ we get that $u_{i+1}=u_{i-1}=0$. By induction we get that $u_k=0$, $k\in\Z$. It means, that every nontirivial critical point of $J$ is a positive solution to problem \eqref{eq}.
	\end{proof}
	
	\section{Proofs of main Theorems}
	\subsection{Proofs of Theorems \ref{tw1}, \ref{tw2} }

\begin{proof}[Proof of Theorem \ref{tw1}]
	Let $\Phi ,\Psi $ and $J$ as in \eqref{defJ}. By Proposition \ref{pom} we need to find a sequence $\{u^{n}\}$ of critical points of $J$ whose norms tend to infinity. 
	
	Let $\{c_{n}\},\{d_{n}\}$ be sequences which satisfy assumption $(F_3^\infty)$. %From $(F_3^\infty)$ and $(F_2)$ there exists $r<0$ such that
	%\begin{equation}\label{pom-1}
	%\sum_{k\in\Z}\max_{t\in[r,d_n]}\abs{F^+_k(t)}\le\sum_{k\in\Z}\max_{t\in[r,d_n]}\abs{F_k(t)}<\infty, \ \textrm{for any} \ n\in\N.
	%\end{equation}
	%Then
	%\begin{equation}\label{pom-1}
	%\max_{t\in[r,d_n]}\abs{\tilde{F}_\cdot(t)}\in l_1, \ \textrm{for any} \ n\in\N.
	%\end{equation}
	%where $\tilde{F}_k(t)=\int^t_0\tilde{f}_k(u)\;du$.
	
	We define the set 
	\begin{align*}
	W_{n}=\left\{ u\in E:\abs{u_k}\leq d_{n}\text{ for every }k\in \Z\right\},
	\end{align*}
	for every $n\in \N$.
	
	We have divided the rest of the proof into few steps.
	
	%\begin{claim}
		%\bigskip \label{c1}
	\textbf{Step 1.} Firstly, we show that, for every $n\in\N$, the functional $J$ is bounded from below on $W_{n}$ and its infimum on $W_{n}$ is attained.
	%\end{claim}
	%\bigskip 
	
	Let $n\in\N$. Note that 
	\begin{equation}\label{pom-1}
	\abs{F^+_k(t)}\le\int^{\abs{t}}_0\abs{f^+_k(s)}\,ds\le \abs{t}\cdot\max_{\abs{u}\le \abs{t}}\abs{f^+_k(u)}
	\end{equation}
	for any $t\in\R$. By the definition of $\Phi$ and \eqref{pom-1} we have 
	$$
	J(u)\ge -\Psi(u)=- \sum_{k\in\Z}F^+_k(u_k)\ge- d_n\sum_{k\in\Z}\max_{\abs{u}\le d_n}\abs{f^+_k(u)}>-\infty,
	$$
	%$$
	%\tilde{J}(u)\ge -\tilde{\Psi}(u)=- \sum_{k\in\Z}\tilde{F}_k(u_k)\ge- \sum_{k\in\Z}\max_{t\in[r,d_n]}\tilde{F}_k(t)>-\infty,
	%$$
	for $u\in W_{n}$. Thus, $J$ is bounded from below on $W_{n}$. Let $\eta _{n}=\inf_{W_{n}}J$ and $\{\tilde{u}^l\}$ be a sequence in $E$ such that $\eta _{n}\leq J(\tilde{u}^{l})\leq \eta _{n}+\frac{1}{l}$ for all $l\in\N$. 
	\par\noindent
	By Propositions \ref{osz-Orl} and \ref{osz-Phi}, if $\norm{\tilde{u}^l}_E\ge1$, then
	$$%\begin{align*}
		\tfrac{1}{p^+}\norm{\tilde{u}^{l}}_E^{p^-}\le\tfrac{1}{p^+}\rho_E(\tilde{u}^{l})\le\Phi(\tilde{u}^{l})=J(\tilde{u}^l)+\sum_{k\in\Z}F^+_k(\tilde{u}^{l}_k)\le \eta _{n}+1+ d_n\sum_{k\in\Z}\max_{\abs{t}\le d_{n}}\abs{f_k(t)}.
	$$%\end{align*}%
%\begin{align*}
	%	\tfrac{1}{p^+}\norm{\tilde{u}^{l}}_E^{p^-}\le\tfrac{1}{p^+}\rho_E(\tilde{u}^{l})\le\Phi(\tilde{u}^{l})=\tilde{J}(\tilde{u}^l)+\sum_{k\in\Z}\tilde{F}_k(\tilde{u}^{l}_k)\le\eta _{n}+1+ \sum_{k\in\Z}\max_{t\in [r,d_{n}]}\tilde{F}_k(t), %\\
		%&\eta _{n}+1+ \sum_{k\in\Z}\max_{t\in [r,d_{n}]}\tilde{F}_k(t),
	%\end{align*}%
%and similarly, if $\norm{\tilde{u}^l}_E<1$, then
	%$$%\begin{align*}
	%\tfrac{1}{p^+}\norm{\tilde{u}^{l}}^{p^+}_E\le\tfrac{1}{p^+}\rho_E(\tilde{u}^{l})\le\Phi(\tilde{u}^{l})\le%=J(\tilde{u}^l)+\sum_{k\in\Z}F_k(\tilde{u}^{l}_k) \\
		%\eta _{n}+1+ \sum_{k\in\Z}\max_{t\in [r,d_{n}]}F_k(t).
	%$$%\end{align*}
	%$$%\begin{align*}
	%	\tfrac{1}{p^+}\norm{\tilde{u}^{l}}^{p^+}_E\le\tfrac{1}{p^+}\rho_E(\tilde{u}^{l})\le\Phi(\tilde{u}^{l})=\tilde{J}(\tilde{u}^l)+\sum_{k\in\Z}\tilde{F}_k(\tilde{u}^{l}_k)\le% \\	\eta _{n}+1+ \sum_{k\in\Z}\max_{t\in [r,d_{n}]}\tilde{F}_k(t).
	%$$%\end{align*}%
	It means that $\{\tilde{u}^{l}\}$ is bounded in $E$. Taking into account that $E$ is a reflexive Banach space and $W_n$ is a weakly closed subset of $E$, we get that $\{\tilde{u}^{l}\}$ weakly converges in $E$ to some $u^{n}\in W_{n}$. By the sequentially weakly lower semicontinuity of $J$ we conclude that $J(u^{n})=\eta _{n}=\inf_{W_{n}}J$. 
	%This proves Claim \ref{c1}.
	
%	\begin{claim}
		%\bigskip \label{c2}
		\textbf{Step 2.} Now, we show that if %For every $n\in \N$, let $u^{n}\in W_{n}$ be such that 
		$u^{n}=\argmin_{W_{n}}J$, $n\in\N$, then, $0\leq u^{n}_k\leq c_{n}$ for all $k\in\Z$.
	%\end{claim}
	
	%\bigskip 
	Fix $n\in\N$. Let $K=\{k\in \Z:u^{n}_k\notin [ 0,c_{n}]\}$ and suppose that $K\neq \emptyset$. Put
	\begin{equation*}
	K_{-}=\{k\in K:\ u^{n}_k<0\}\qquad \mbox{and}\qquad K_{+}=\{k\in K:\ u^n_k>c_{n}\}.
	\end{equation*}%
	Thus, $K=K_{-}\cup K_{+}$.
	
	We define the truncation function $\gamma :\R\to\R$ by formula $\gamma (s)=\min \{s^{+},c_{n}\}$. Now, put $
	w^n=\gamma \circ u^{n}$. Clearly, $w^{n}\in E$ and $w^{n}_k\in[ 0,c_{n}]$ for every $k\in \Z$. Thus $w^{n}\in W_{n}$.
	\begin{align*}
	&J(w^{n})-J(u^{n})=\sum_{k\in \Z}\tfrac{a_k}{p_k}\left( \abs{\nabla^+ w^n_k}^{p_k}-\abs{\nabla^+ u^n_k}^{p_k}\right) 
	\\
	&+\sum_{k\in \Z}\tfrac{a_k}{p_k}\left( \abs{w^n_k}^{p_k}-\abs{u^n_k}^{p_k}\right) -\sum_{k\in\Z}[F^+_k(w^{n}_k)-F^+_k(u^{n}_k)] =:I_{1}+I_{2}- I_{3}.
	\end{align*}
	Since $\gamma $ is a Lipschitz function with Lipschitz-constant equal to 1, we get
	$$
	I_{1} =\sum_{k\in \Z}\tfrac{a_k}{p_k}\left( \abs{\nabla^+ w^n_k}^{p_k}-\abs{\nabla^+ u^n_k}^{p_k}\right)=\sum_{k\in \Z}\tfrac{a_k}{p_k}\left( \abs{w^n_{k+1}-w^n_{k}}^{p_k}-\abs{u^n_{k+1}- u^n_k}^{p_k}\right) \le0. 
	$$
	Moreover, as $w^{n}_k=u^{n}_k$ for all $k\in\Z \setminus K$, $w^{n}_k=0$ for all $k\in K_{-}$, and $w^{n}_k=c_{n}$ for
	all $k\in K_{+}$, we have 
	$$%\begin{align*}
	I_{2} =\sum_{k\in \Z}\tfrac{b_k}{p_k}\left( \abs{w^{n}_k}^{p_k}-\abs{u^{n}_k}^{p_k}\right) =%\sum_{k\in K}\tfrac{b_k}{p_k}\left( \abs{w^{n}_k}^{p_k}-\abs{w^{n}_k}^{p_k}\right) = 
	\sum_{k\in K_-}\tfrac{b_k}{p_k}\left( -\abs{u^{n}_k}^{p_k}\right) +\sum_{k\in K_+}\tfrac{b_k}{p_k}\left( c_n^{p_k}-\abs{u^{n}_k}^{p_k}\right) \le0.
	$$%\end{align*}
	Next, we estimate $I_{3}$. Firstly, by the definition we have $F^+_k(s)=0$ for $s\leq 0$, $k\in \Z$, and consequently $\sum_{k\in K_{-}}(F^+_k(w^{n}_k)-F^+_k(u_{n}^k))=0$. As $u^n\in W_n$ we have $[c_n,u^n_k]\subset[c_n,d_n]$ for $k\in K_+$, and by $(F_3^\infty)$,  we have $f_k(t)\le0$, for $t\in[c_n,u^n_k]$, hence
	$$
	\sum_{k\in K_+}(F^+_k(c_n)-F^+_k(u^n_k))=\sum_{k\in K_+}(F_k(c_n)-F_k(u^n_k))=-\sum_{k\in K_+}\int\limits^{u^n_k}_{c_n}f_k(t)\,dt\ge0.
	$$
	Consequently,
	$$
	I_{3} =\sum_{k\in \Z}[F^+_k(w_{n}^k)-F^+_k(u_{n}^k)]=\sum_{k\in K_+}[F_k(w_{n}^k)-F_k(u_{n}^k)]\ge0.\nonumber
	 $$
	%\begin{align*}
	%&I_{3} =\sum_{k\in \Z}[\tilde{F}_k(w_{n}^k)-\tilde{F}_k(u_{n}^k)]=\sum_{k\in K_+}[\tilde{F}_k(w_{n}^k)-\tilde{F}_k(u_{n}^k)]\ge0\nonumber
	 %\\
	%&\sum_{k\in K_{-}}[-F_k(u_{n}^k)]+\sum_{k\in K_{+}}[F_k(w_{n}^k)-F_k(u_{n}^k)]\geq ?????????????? 0.  \label{I3}
	%\end{align*}
	Combining above inequalities we get that 
	\begin{equation*}
	J(w^{n})-J(u^{n})\leq 0.
	\end{equation*}%
	But $J(w^{n})\geq J(u^{n})=\inf_{W_{n}}J$, since $w^{n}\in W_{n}$. So, every term in $J(w^{n})-J(u^{n})$ is equal to zero. In particular, $I_{2}=0$, hence
	\begin{equation*}
	\sum_{k\in K_{-}}\tfrac{b_k}{p_k}\abs{u^{n}_k}^{p_k}=\sum_{k\in
		K_{+}}\tfrac{b_k}{p_k}\left(c_{n}^{p_k}-\abs{u^{n}_k}^{p_k}\right)=0.
	\end{equation*}%
	By $(A)$, we get that $u^{n}_k=0$ for every $k\in K_{-}$ and $u^{n}_k=c_{n}$
	for every $k\in K_{+}$. By the definition of sets $K_{-}$ and $K_{+}$, we
	must have $K_{-}=K_{+}=\emptyset $, this contradicts to $K_{-}\cup K_{+}=K\neq
	\emptyset $; therefore $K=\emptyset$. % This proves Claim \ref{c2}.

	%\begin{claim}
		%\bigskip \label{c3}
	\textbf{Step 3.} Now, we show that if $u^{n}=\argmin_{W_{n}}J$, then $u^{n}$ is a critical point of $J$ for every $n\in\N$.%\bigskip
	%\end{claim}
	
	It is sufficient to show that $u^{n}$ is a local minimum point of $J$ in $E$. Assuming the contrary, consider a sequence $\{v^{i}\}\subset E$ which converges to $u^{n}$ and $J(v^{i})<J(u^{n})=\inf_{W_{n}}J$ for all $i\in \N$. From this inequality it follows that $v^{i}\notin W_{n}$ for any $i\in\N$. Since $v^{i}\to u^{n}$ in $E$, due to \eqref{alfa-osz}, $v^{i}\to u^{n}$ in $l_{\infty}$ as well. Choose a positive $\delta$ such that $\delta <\frac{1}{2}(d_{n}-c_{n})$. Then, there exists $i_{\delta }\in \N$ such that $\norm{v^{i}-u^{n}}_{\infty }<\delta$ for every $i\geq i_{\delta}$. By using previous Step and taking into account the choice of the number $\delta$, we conclude that $\abs{v^{i}_k}\le d_{n}$ for all $k\in \Z$ and $i\geq i_{\delta}$, which contradicts to the fact $v^{i}\notin W_{n}$. %This proves Claim \ref{c3}.
	
	%\bigskip
	%\begin{claim}
		%\bigskip \label{inf} 
	\textbf{Step 4.}	Now we show that $\lim_{n\to+\infty }\eta _{n}=-\infty$, where $\eta _{n}=\inf_{W_{n}}J$, $n\in\N$.
	%\end{claim}
	
	Firstly, we assume that $(F_5^\infty)$ holds. Without lost of generality we can assume that
	$$
	\limsup\limits_{(k,t)\rightarrow (+\infty,+\infty)}\frac{F_k(t)}{\left( a_{k-1}+a_k+b_k\right) t^{p^+}}>\tfrac{1}{p^-}.
 $$
	There exists $\e>0$ such that
	$$
	\limsup\limits_{(k,t)\rightarrow (+\infty,+\infty)}\frac{F_k(t)}{\left( a_{k-1}+a_k+b_k\right) t^{p^+}}>\tfrac{1}{p^-}+\e.
 $$
	Hence we get the existence of sequences $\{k_n\}_{n\in\N}$, $\{t_n\}_{n\in\N}$ such that $k_n\to+\infty$, $t_n\to\infty$, $t_n\ge1$ and
	\begin{equation*}
	F_{k_n}(t_{n})>(\tfrac{1}{p^-}+\e)(a_{k_{n}-1}+a_{k_{n}}+b_{k_{n}})t_{n}^{p^+},
	\end{equation*}
	for all $n\in\N$. Without lost of generality, we may assume that $d_{n}\geq t_{n}\geq 1$ for all $n\in \N$. We define a sequence $\{w^{n}\}$ such that, for every $n\in\N$, $w^{n}_{k}=t_{n}$, if $k=k_n$ and $w^{n}_k=0$ if $k\in\Z\backslash \{k_{n}\}$. It is clear that $w^{n}\in W_{n}.$ Hence
	\begin{align*}
	&J(w^{n})=\tfrac{a_{k_n}}{p_{k_n}}t_n^{p_{k_n}}+\tfrac{a_{k_n-1}}{p_{k_n-1}}t_n^{p_{k_n-1}}+\tfrac{b_{k_n}}{p_{k_n}}t_n^{p_{k_n}}-F^+_{k_n}(t_n) \\
		&\le\tfrac{1}{p^-}(a_{k_n}+a_{k_n-1}+b_{k_n})t_{n}^{p^+}-F_{k_n}(t_n)<-\e(a_{k_{n}-1}+a_{k_{n}}+b_{k_{n}})t_{n}^{p^+}
			\end{align*}
	which gives $\lim_{n\to +\infty }J(w^{n})=-\infty $. %Similarly we argue if when the second part of $(F^\infty_5)$ holds. 
	Now, assume that $(F_6^\infty)$ holds, i.e.% Without lost of generality we can assume that
	$$
	\sup_{k\in\Z}\left(\limsup\limits_{t\to +\infty}\frac{F_k(t)}{\left( a_{k-1}+a_k+b_k\right) t^{p^+}}\right)>\tfrac{1}{p^-}.
 $$
%	Let $\e>0$ such that
	%$$
	%\sup_{k\in\Z}\left(\limsup\limits_{t\to +\infty}\frac{F_k(t)}{\left( a_{k-1}+a_k+b_k\right) t^{p^+}}\right)>\tfrac{1}{p^-}+\e.
 %$$
	There exist $\e>0$ and $k_{0}\in\Z$ such that%
	\begin{equation*}
	\limsup\limits_{t\to +\infty }\frac{F_{k_0}(t)}{(a_{k_{0}-1}+a_{k_{0}}+b_{k_{0}})t^{p^+}}>\tfrac{1}{p^-}+\e.
	\end{equation*}%
	Then, there exists a sequence of real numbers $\{t_{n}\}$ such that $\lim_{n\to +\infty }t_{n}=+\infty$ and 
	\begin{equation*}
	F_{k_0}(t_{n})>\left(\tfrac{1}{p^-}+\e\right)(a_{k_{0}-1}+a_{k_{0}}+b_{k_{0}})t_{n}^{p^+}
	\end{equation*}
	for all $n\in \N$. Without lost of generality, we may assume that $d_{n}\geq t_{n}\geq
	1 $ for all $n\in \N$. Thus, take in $E$ a sequence $\{w^{n}\}$ such that, for every $n\in \N$, $w^{n}_{k_{0}}=t_{n}$ and $w^{n}_k=0$ for every $k\in \Z\backslash \{k_{0}\}$. Then, $w^{n}\in W_{n}$ and
	\begin{align*}
	&J(w^{n}) =\tfrac{a_{k_0}}{p_{k_0}}t_n^{p_{k_0}}+\tfrac{a_{k_0-1}}{p_{k_0-1}}t_n^{p_{k_0-1}}+\tfrac{b_{k_0}}{p_{k_0}}t_n^{p_{k_0}}-F^+_{k_0}(w^{n}_{k_0}) \\
		&\le\tfrac{1}{p^-}(a_{k_0}+a_{k_0-1}+b_{k_0})t_{n}^{p^+}-F_{k_0}(t_n)<-\e(a_{k_{0}+1}+a^n_{k_{0}}+b_{k_{0}})t_{n}^{p^+}
		\end{align*}%
	which gives $\lim_{n\rightarrow +\infty }J(w^{n})=-\infty $.% This proves Claim \ref{inf}.
	
	Now, we are ready to end the proof of Theorem \ref{tw1}. From Proposition \ref{pom} and the above Steps %Claims \ref{c2}--\ref{inf}, up to a subsequence, 
	we have infinitely many pairwise distinct positive homoclinic solutions $u^{n}$ to \eqref{eq} with $u^{n}\in W_{n}$. To finish the proof, we will prove that $\norm{ u^{n}}_{E}\rightarrow +\infty $ as $n\rightarrow +\infty$. On the contrary suppose, that there exists a subsequence $\{u^{n_{i}}\}$ of $\{u^{n}\}$ which is bounded in $E$. From reflexivity of $E$ there exists $u^0\in E$ such that $u^{n_i}\rightharpoonup u^0$, as $i\to\infty$, by choosing a subsequence, if necessary. From sequentially weakly lower semicontinuity of $J$ on $E$ we get
	$$
	-\infty=\liminf_{i\to\infty}J(u^{n_{i}})\ge J(u^0),
	$$
	a contradiction. %This contradiction concludes our proof.
\end{proof}
\begin{proof}[Proof of Theorem \ref{tw2}] 
In analogously way to the proof of Theorem \ref{tw1} we can prove the existence of a sequence $\{u^n\}$ of positive solutions to \eqref{eq} such that $u^n=\argmin_{W_n}J$, where $W_n=\{u\in E:\norm{u}_{\infty}\le d_n\}$ and $d_n$ satisfies $(F^0_3)$. Moreover, $0\le u^n_k\le c_n$ for any $n\in\N$, $k\in\Z$. As in Step 4 from Theorem \ref{tw1} we get that $J(u^n)<0$, $n\in\N$. Moreover, we get that 
$$
0>J(u^n)\ge-c_n\sum_{k\in\Z}\max_{\abs{t}\le d_1}\abs{f_k(t)}.
$$
Above and $(F_2)$, $(F_3^0)$ imply that $\lim_{n\to\infty}J(u^n)=0$. To finish the proof, notice that 
$$
0\le\tfrac{1}{p^+}\rho_{E}(u^n)\le\Phi(u^n)=J(u^n)+\sum_{k\in\Z}F^+_k(u^n_k)<c_n\sum_{k\in\Z}\max_{\abs{t}\le d_1}\abs{f_k(t)}
$$
for $n\in\N$. Hence $\norm{u}_E\to0$, as $n\to\infty$.
\end{proof}

\subsection{Proofs of Theorems \ref{tw2'}, \ref{tw1'}}
We start this section with presentation of main tool in our proofs.
\begin{theorem}\label{cpt}\cite{BMB}, \cite{R}
Let $(E,\left\Vert \cdot \right\Vert )$ be a reflexive real Banach space, let $\Phi ,\Psi :E\to\R$ be two continuously differentiable functionals with $\Phi $ coercive, i.e. $\lim_{\norm{u}\to \infty }\Phi (u)=+\infty ,$ and a sequentially weakly lower semicontinuous functional and $\Psi $ a sequentially weakly upper semicontinuous functional. For every $r>\inf_{E}\
\Phi $, let us put
\begin{equation*}
\varphi (r):=\inf_{u\in \Phi ^{-1}((-\infty ,r))}\frac{\left( \sup_{v\in
\Phi ^{-1}((-\infty ,r))}\Psi (v)\right) -\Psi (u)}{r-\Phi (u)}
\end{equation*}%
and%
\begin{equation*}
\delta :=\liminf_{r\rightarrow (\inf_{E}\Phi )^{+}}\varphi (r),\quad \gamma :=\liminf_{r\to+\infty}\varphi(r).
\end{equation*}%
Let $J_{\lambda }:=\Phi (u)-\lambda \Psi (u)$ for all $u\in E$. Then
\begin{itemize}
\item[i)] If $\delta<+\infty $ then, for each $\lambda \in \left( 0,\frac{1}{\delta }\right)$,
the following alternative holds: either
\begin{itemize}
\item[$i_a)$] there is a global minimum of $\Phi $ which is a local minimum of $%
J_{\lambda }$, or 
\item[$i_b)$]there is a sequence $\{u^{n}\}$ of pairwise distinct critical points of $J_{\lambda }$, with $\lim_{n\rightarrow +\infty }\Phi
(u^{n})=\inf_{E}\Phi $, which weakly converges to a global minimum of $\Phi $.
\end{itemize}
\item[ii)] If $\gamma<+\infty $ then, for each $\lambda \in \left( 0,\frac{1}{\gamma}\right)$, the following alternative holds: either
\begin{itemize}
\item[$ii_a)$] $J_{\lambda }$ has a global minimum, or
\item[$ii_b)$]  there is a sequence $\{u^{n}\}$ of critical points of $J_{\lambda }$, with $\lim_{n\rightarrow +\infty }\Phi
(u^{n})=+\infty$.
\end{itemize}
\end{itemize}
\end{theorem}

\begin{proof}[Proof of Theorem \ref{tw2'}]  %[Proof]
%To apply Proposition \ref{pom}, we need to have a nonlinearity which satisfies condition $(\hat{F}_{1}).$ Let $T_{0}>0$ be a number satisfying $(F_{1})$. Define the truncation function
%\begin{equation*}
%\bar{f}(k,s)=\left\{ 
%\begin{array}{ll}
%f(x,-T_{0}), & s\leq -T_{0}\text{ and }k\in \Z, \\ 
%f(x,s), & -T_{0}\leq s\leq T_{0}\text{ and }k\in \Z, \\ 
%f(x,T_{0}), & s\geq T_{0}\text{ and }k\in \Z.
%\end{array}%
%\right.
%\end{equation*}%
%and consider the problem 
%\begin{equation}
%\left\{ 
%\begin{array}{l}
%-\Delta \left( a(k)\phi _{p}(\Delta u(k-1))\right) +b(k)\phi_{p}(u(k))=\lambda \bar{f}(k,u(k)) \\ 
%u(k)\rightarrow 0.%
%\end{array}%
%\right.  \label{eq2}
%\end{equation}%
%\ Clearly, if $u$ is a solution of the problem (\ref{eq2}) with $\left\Vertu\right\Vert _{\infty }\leq T_{0}$, then it is also a solution to problem \eqref{eq}, so it is enough to show that the problem \eqref{eq2} admits a nonzero sequence of solutions in $X$ whose norms tend to zero.
Let $\Phi$ be as in \eqref{Phi} and $\Psi$ be as in \eqref{Psi}. By Lemma \ref{phi-C1}, $\Phi $ is a continuously differentiable, sequentially weakly lower semicontinuous functional on $E$ and by Proposition \ref{osz-Phi}  it is easy to see that $\Phi$ is coercive. $\Psi $ is a continuously differentiable and sequentially weakly upper semicontinuous
functional on $E$. Let $J$ be as in \eqref{defJ}. To apply Theorem \ref{cpt} to the function $J=J_1$, we show that $\delta<1$.  By the definition of $\varphi$ we get that $\delta\ge0$. Let
$$
A:=\liminf\limits_{t\to0^+}\frac{\sum_{k\in\Z}\max_{\abs{\xi}\le t}F_k(\xi)}{t^{p^+}}. %<\frac{1}{p^+\alpha^{p^+}}
$$
By the definition of $F_k$ and assumption $(F_4^0)$ we get that $A\in[ 0, \tfrac{1}{p^+\alpha^{p^+}})$, where $\alpha$ is given by \eqref{alfa}. Let $\{c_{m}\}\subset(0,\alpha)$ be a sequence such that $\lim_{m\rightarrow \infty }c_{m}=0$ and%
\begin{equation}\label{A}
\lim_{m\rightarrow +\infty }\frac{\sum_{k\in\Z}\max_{\left\vert \xi \right\vert \leq c_{m}}F_k(\xi )}{c_{m}^{p^+}}=A.
\end{equation}%
Set 
\begin{equation*}
r_{m}:=\frac{1}{p^+}\left(\frac{c_m}{\alpha}\right)^{p^+}
\end{equation*}%
for every $m\in \N$. Then, $r_m\in(0,1)$ and  $r_m\to0$, as $m\to\infty$. By Proposition \ref{osz-Phi}, if $v\in E$ and $\Phi (v)<r_{m}$ for $m\in\N$, then we have $\norm{v}_{E}\le 1$ and hence %from Proposition \ref{osz-Phi}
\begin{equation*}
\norm{v}_{\infty }\leq \alpha (p^+\Phi(v))^{\tfrac{1}{p^+}} \leq \alpha (p^+r_m)^{\tfrac{1}{p^+}}=c_{m}
\end{equation*}%
%\newline
and %which implies
\begin{equation}
\Phi ^{-1}\left( \left( -\infty ,r_{m}\right) \right) \subset \left\{ v\in
E:\norm{v}_{\infty }\leq c_{m}\right\} .
\end{equation}%
From above and $\Phi (0)=\Psi (0)=0$ we have%
$$
\varphi (r_{m})\leq \frac{\sup_{\Phi (v)<r_{m}}\sum_{k\in \Z}
F^+_k(v_k)}{r_{m}}\leq \frac{\sum_{k\in\Z}\max_{\abs{t}\leq c_{m}}F_k(t)}{r_{m}} =p^+\alpha^{p^+}\cdot \frac{\sum_{k\in \Z}\max_{\abs{t}\leq c_{m}}F_k(t)}{c_{m}^{p^+}}
$$
for $m\in\N$. Hence 
\begin{equation*}
\delta \leq \lim_{m\to +\infty }~\varphi (r_{m})\leq p^+\alpha^{p^+}A<1.
\end{equation*}

Now, we show that the point $(i_a)$ in Theorem \ref{cpt} does not hold, i.e. that the unique global minimum 0 of $\Phi $\ is not a local minimum of $J$. 

Firstly, we assume that $(F_5^0)$ is satisfied. Without lost of generality, there exist sequence $\{k_n\}\subset\N$ and sequence $\{t_n\}\subset(0,1)$ such that $k_n\to\infty$, $t_n\to0^+$ and
\begin{equation}\label{osz-2}
F_{k_n}(t_n)>\tfrac{1}{p^-}\left(a_{k_n-1}+a_{k_n}+b_{k_n}\right)t_n^{p^-},
\end{equation}
for any $n\in\N$. Now we define a sequence $\{s^n\}$, where $s^n_k=t_n$ for $k=k_n$, and $s^n_k=0$ otherwise. 
%$$
%s^n_k=\begin{cases}
%t_n&\mbox{for}\ k=k_n
%\\
%0&\mbox{for}\ k\neq k_n
%\end{cases}.
%$$
It is easy to see that for any $n\in\N$, $s^n\in E$,
$$
J(s^n)=\tfrac{a_{k_n}}{p_{k_n}}t_n^{p_{k_n}}+\tfrac{a_{k_n-1}}{p_{k_n-1}}t_n^{p_{k_n}-1}+\tfrac{b_{k_n}}{p_{k_n}}t_n^{p_{k_n}}-F^+_{k_n}(t_n)\le \tfrac{1}{p^-}\left(a_{k_n-1}+a_{k_n}+b_{k_n}\right)t_n^{p^-}-F_{k_n}(t_n)<0,
$$
and $\norm{s^n}_{\infty}\to0$. We prove that $\norm{s^n}_{E}\to0$. Notice that form \eqref{osz-2} we have
$$%\begin{equation}\label{norm}
\norm{s^n}_{E}=\tfrac{a_{k_n}}{p_{k_n}}t_n^{p_{k_n}}+\tfrac{a_{k_n-1}}{p_{k_n-1}}t_n^{p_{k_n}-1}+\tfrac{b_{k_n}}{p_{k_n}}t_n^{p_{k_n}}\le \tfrac{1}{p^-}\left(a_{k_n-1}+a_{k_n}+b_{k_n}\right)t_n^{p^-}<F_{k_n}(t_n).
$$%\end{equation}
Without lost of generality, we may assume that $c_{m}\geq t_{m}$ for all $m\in \N$. Since $\lim_{m\rightarrow +\infty }c_{m}=0$ and the limit in \eqref{A} is finite, we get
\begin{equation*}
\lim_{m\to +\infty}\sum_{k\in\Z}\max_{\abs{\xi}\leq c_{m}}F_k(\xi )=0.
\end{equation*}
Therefore, we get $\lim_{m\rightarrow +\infty }\left( \max_{\abs{ \xi}
\leq c_{m}}F_{k}(t)\right) =0$ uniformly for all $k\in \Z$. Hence $\lim_{n\rightarrow +\infty
}F_{k_{n}}(t_{n})=0$ and so $\lim_{n\rightarrow +\infty }\norm{s^{n}}_E =0.$

Now, we assume that $(F_6^0)$ holds. There exist $k_0\in\Z$ and sequence $\{t_n\}\subset(0,1)$ such that $t_n\to0^+$ and
$$%\begin{equation}\label{osz-2}
F_{k_0}(t_n)>\tfrac{1}{p^-}\left(a_{k_0-1}+a_{k_0}+b_{k_0}\right)t_n^{p^-},
$$%\end{equation}
for any $n\in\N$. Moreover, we define a sequence $\{s^n\}$, where $s^n_k=t_n$ for $k=k_0$, and $s^n_k=0$ otherwise.  
%$$
%s^n_k=\begin{cases}
%t_n&\mbox{for}\ k=k_0
%\\
%0&\mbox{for}\ k\neq k_0
%\end{cases}.
%$$
As in the previous part, we can show that $J(s^n)<0$, $n\in\N$ and $\norm{s^n}_E\to0$ , as $n\to\infty$. 

From above it follows that 0 is not a local minimum of $J$ and, by $(i_b)$ from Theorem \ref{cpt}, there is a sequence $\{u^{n}\}$ of pairwise
distinct critical points of $J$ with $\lim_{n\rightarrow +\infty}\Phi (u^{n})=\inf_{E}\Phi $. The last means that $0=\inf_{E}\Phi =$ $\lim_{n\to +\infty }\Phi (u^{n})\ge\lim_{n\to +\infty } \frac{1}{p^+}\rho_E(u^n)$, and by Proposition \ref{osz-Phi}, $\{u^{n}\}$\ strongly converges to zero. %The proof is complete.
\end{proof}

\begin{proof}[Proof of Theorem \ref{tw1'}] 

In analogously way to the proof of Theorem \ref{tw2'} we can prove that $\gamma<1$ and we exclude condition $(ii_a)$ in Theorem \ref{cpt}. It  means that $J$ does not have a global minimum and, by $(ii_b)$ from Theorem \ref{cpt}, there is a sequence $\{u^{n}\}$ of critical points of $J$ with $\lim_{n\rightarrow +\infty}\Phi (u^{n})=\infty$. %$\lim_{n\to +\infty }\Phi (u^{n})\le\lim_{n\to +\infty } \frac{1}{p^-}\rho_E(u^n)$, and 
By Propositions \ref{osz-Orl} and \ref{osz-Phi} we have $\lim_{n\to\infty}\norm{u^{n}}_E=\infty$. The proof is complete.
\end{proof}

\subsection{Examples}
\begin{example}
Consider a problem 
\begin{equation}
\left\{ 
\begin{array}{ll}
-\nabla^- \left( \abs{\nabla^+ u_k}^{p_k-2}\nabla^+ u_k\right) +\abs{k}\abs{u_k}^{p_k-2}u_k=f_k(u_k)& \mbox{for
all $k\in\mathbb{Z}$} \\ 
u(k)\rightarrow 0 & \mbox{as $|k|\to \infty$},%
\end{array}%
\right.  \label{ex}
\end{equation}%
where $\{p_k\}_{k\in\Z}$ is any bounded sequence such that $1<p^-=\inf_{k\in\Z}p_k\le \sup_{k\in\Z}p_k=p^+$. Fix $q\ge2$ such that $\tfrac{p^++1}{p^-}<q<p^++1$. For $k\in\Z$, $f_k:\R\to\R$ is defined by 
\begin{equation}
f_k(s)=e_{k}\left( d_{k}-c_{k}-2\left\vert s-\tfrac{1}{2}\left( c_{k}+d_{k}\right)
\right\vert \right) \cdot \mathbf{1}_{ [c_{k},d_{k}]}(s)
\label{fun}
\end{equation}%
with sequences $\{c_{m}\},\{d_{m}\},\{e_{m}\},\{h_{m}\}$ defined by 
\begin{equation}\label{wsp}
c_{m}=\frac{1}{2^{q^{2m}}}, \quad d_{m}=\frac{1}{2^{q^{2m-1}}}, \quad h_{m}=\frac{1}{2^{(p^++1)q^{2m-2}}}, \quad e_{m}=\frac{2h_{m}}{(d_{m}-c_{m})^{2}},\quad m\in\N.
\end{equation}
%\begin{equation}
%\left\{ 
%\begin{array}{ll}
%c_{m}=\frac{1}{2^{t^{2m}}} & \text{for }m\in \N; \\ 
%d_{m}=\frac{1}{2^{t^{2m-1}}} & \text{for }m\in \N; \\ 
%h_{m}=\frac{1}{2^{(p^++1)t^{2m-2}}} & \text{for }m\in\N ; \\ 
%e_{m}=\frac{2h_{m}}{(d_{m}-c_{m})^{2}} & \text{for }m\in\N.
%\end{array}%
%\right.  \label{seq}
%\end{equation}%
Here $\mathbf{1}_{A }$ is the indicator of $A$. It is easily seen that $(A)$ is satisfied, $f_k$ are continuous for $k\in\Z$ and $(F_{1})$ is satisfied. Notice that $\{d_m\}$ is a decreasing sequence of positive numbers which is convergent to 0 and $f_k(t)=0$ for $t\notin[c_k,d_k]$, hence it is enough to check $(F_2)$ for $T=d_1$. For $T=d_1$ we get that
$$
\sum_{k\in\Z}\max_{\abs{t}\le d_1}\abs{f_k(t)}=%\sum_{k\in\N}\abs{f_k(\tfrac{1}{2}(d_k+c_k))}=
\sum_{k\in\N}\tfrac{2h_k}{d_k-c_k}=\sum_{k\in\N}\frac{2^{1-q^{2m-1}\left((p^++1)q^{-1}-1\right)}}{1-2^{q^{2m}\left(q^{-1}-1\right)}}\le C\sum_{k\in\N}2^{1-q^{2m-1}\left((p^++1)q^{-1}-1\right)}<\infty,
$$
for $C>0$ such that $1-2^{q^{2m}\left(q^{-1}-1\right)}>C^{-1}$ for any $m\in\N$. Notice that $F_k(d_{k})=\int_{0}^{d_{k}}f_k(t)dt=\int_{c_{k}}^{d_{k}}f_k(t)dt=h_{k}$ and
\begin{align}
&\liminf_{t\rightarrow 0^{+}}\frac{\sum_{k\in\Z }\max_{\abs{\xi} \leq t}F_k(\xi)}{t^{p^+}} \leq
\lim_{m\rightarrow +\infty }\frac{\sum_{k\in \Z}\max_{\abs{\xi} \leq c_{m}}F_k(\xi )}{c_{m}^{p^+}}=\lim_{m\rightarrow +\infty }\frac{\sum_{k=m+1}^{\infty }h_k}{c_{m}^{p^+}}  \label{A}
\\
&%=\lim_{m\rightarrow +\infty }\frac{\sum_{k=m+1}^{\infty }h_{k}}{c_{m}^{p^+}}
\le\lim_{m\rightarrow +\infty }\frac{2h_{m+1}}{c_{m}^{p^+}}= \lim_{m\rightarrow +\infty }2^{-q^{2m}+1}=0  \notag
\end{align}
and 
\begin{equation}
\limsup\limits_{(k,t)\rightarrow (+\infty ,0^{+})}\frac{F_k(t)}{(2+k)t^{p^-}}\ge\lim_{m\rightarrow +\infty }\frac{F_m(d_{m})}{(2+m)d_{m}^{p^-}}
=\lim_{m\to\infty}\frac{2^{q^{2m-1}(p^--q^{-1}(p^++1))}}{2+m}=+\infty.\label{BB} 
%\\
%&\lim_{m\rightarrow +\infty }\frac{2^{2^{(2m-1)}(p^--\tfrac{1}{2}p^+)}}{2+\tfrac{p^-}{2}2^{2^{(2m-1)}(p^--\tfrac{1}{2}p^+)}}=\tfrac{2}{p^-}>\tfrac{1}{p^-}. % \notag
\end{equation}
It means that conditions $(F_{4}^0),(F_{5}^{0})$ are satisfied and by Theorem \ref{tw2'} problem \eqref{ex} with \eqref{fun} and \eqref{wsp} admits a sequence of positive solutions in $E$ whose norms tend to zero. Notice that $(F^0_6)$ is not satisfied in this case.
\end{example}
\begin{remark}\label{rem-1}
 Fix $k_{0}\in \Z$. If we define $f_{k}\equiv0$ for $k\in\Z\setminus\{k_0\}$ and 
$$
f_{k_0}(s)=\sum_{m\in\N}e_m\left( d_{m}-c_{m}-2\left\vert s-\tfrac{1}{2}\left( c_{m}+d_{m}\right)
\right\vert \right) \cdot \mathbf{1}_{[c_{m},d_{m}]}(s)
$$
with sequences $\{c_{m}\},\{d_{m}\},\{e_{m}\},\{h_{m}\}$ defined in \eqref{wsp}, then $(F_{1})$, $(F_{2})$, $(F_{4}^0)$ and $(F_{6}^0)$ are satisfied, but $(F_{5}^{0})$ is not satisfy.
\end{remark}
\begin{remark}\label{rem-2} 
Let us note that Theorem \ref{tw2} and Theorem \ref{tw2'} are independent. Let $p^+>2$. Fix $q\ge2$ such that $\tfrac{p^+}{p^-}<q<p^+$. Let us replace $h_{m}$ in \eqref{wsp} by 
$$
h_{m}=\frac{2}{\alpha^{p^+}p^+2^{p^+q^{2m-4}}},\  \text{for }\ m\in \N.
$$
Then functions $f_k$ given by \eqref{fun} are continuous for $k\in\Z$ and $(F_1)$, $(F_2)$ are satisfied. %It can be seen that the first inequality in \eqref{A} is in fact equality. Then, 
An easy computation shows that 
$$
\liminf_{t\rightarrow 0^{+}}\frac{\sum_{k\in\Z}\max_{\left\vert \xi \right\vert \leq t}F_k(\xi )}{t^{p}}\geq \lim_{m\rightarrow +\infty }\frac{\sum_{k\in \Z}\max_{\abs{\xi} \leq c_{m+1}}F_k(\xi )}{c_{m}^{p^+}}\ge\tfrac{2}{\alpha^{p^+}p^+}
$$
and%
\begin{equation*} 
\limsup\limits_{(k,t)\rightarrow (+\infty ,0^{+})}\frac{F_k(t)}{(2+k)t^{p^-}}=+\infty.
\end{equation*}%
This means that we can not apply Theorem \ref{tw2'}, but Theorem \ref{tw2} works. On the other hand, it is easy to see that we can modify $f_k$ in the way, that for some (or even infinitely many) $k$ we have $f_k(t)>0$ for all $t>0$ and the limits \eqref{A}, \eqref{BB} do not change. Therefore, such  $f_k$ do not satisfy $(F_{3}^0)$ and can not be used in Theorem \ref{tw2}.
\end{remark}
\begin{example}
Consider  problem \eqref{ex} with $\{p_k\}_{k\in\Z}$ where $\{p_k\}_{k\in\Z}$ is any bounded sequence such that $1<p^-=\inf_{k\in\Z}p_k\le \sup_{k\in\Z}p_k=p^+$. Fix $q\ge2$ such that $q>\tfrac{p^+}{p^--1}$. For $k\in\Z$ $f_k:\R\to\R$ is defined by 
\begin{equation}
f_k(s)= e_k\left( d_{k}-c_{k}-2\abs{s-\tfrac{1}{2}\left( c_{k}+d_{k}\right)}\right) \cdot \mathbf{1}_{[c_{k},d_{k}]}(s)
\label{fun1}
\end{equation}%
with sequences $\{c_{m}\},\{d_{m}\},\{e_{m}\},\{h_{m}\}$ defined by 
\begin{equation}\label{wsp1}
c_{m}=2^{q^{2m}}, \quad d_{m}=2^{q^{2m+1}}, \quad h_{m}=2^{(p^--1)q^{2m+2}}, \quad e_{m}=\tfrac{2h_{m}}{(d_{m}-c_{m})^{2}},\quad m\in\N.
\end{equation} 
%$$%\begin{equation}
%\left\{ 
%\begin{array}{ll}
%c_{m}=2^{2^{2m}} & \text{for }m\in \N; \\ 
%d_{m}=2^{2^{2m+1}} & \text{for }m\in \N; \\ 
%h_{m}=2^{(p^--1)t^{2m+2}} & \text{for }m\in\N ; \\ 
%e_{m}=\frac{2h_{m}}{(d_{m}-c_{m})^{2}} & \text{for }m\in\N.
%\end{array}%
%$$%\right  %\label{seq}
%\end{equation}%
It is easily seen that $f_k$ are continuous for $k\in\Z$ and $(F_{1})$ is satisfied. For any $T>0$ the set $\{k\in\Z:\max_{\abs{t}\le T}\abs{f_k(t)}>0\}$ is finite, so $(F_2)$ is satisfied. Notice that $\int_{c_{k}}^{d_{k}}f_k(t)dt=h_{k}$ and
\begin{align*}
&\liminf_{t\rightarrow +\infty}\frac{\sum_{k\in\Z }\max_{\abs{\xi} \leq t}F_k(\xi)}{t^{p^-}} \leq
\lim_{m\rightarrow +\infty }\frac{\sum_{k\in\Z }\max_{\abs{\xi} \leq c_{m}}F_k(\xi )}{c_{m}^{p^-}}=\lim_{m\rightarrow +\infty }\frac{\sum^{m-1}_{k=1}h_k}{c_{m}^{p^-}}  %\label{A}
\\
&\le\lim_{m\rightarrow +\infty }\frac{(m-1)h_{m-1}}{c_{m}^{p^-}}= \lim_{m\rightarrow +\infty }\frac{m-1}{2^{q^{2m}}}=0  %\notag
\end{align*}
and 
$$
\limsup\limits_{(k,t)\rightarrow (+\infty ,+\infty)}\frac{F_k(t)}{(2+k)t^{p^-}}\ge\lim_{m\rightarrow +\infty }\frac{F_m(d_{m})}{(2+m)d_{m}^{p^-}}=\lim_{m\to\infty}\frac{2^{q^{2m+1}((p^--1)q-p^+)}}{2+m}=+\infty.%\label{BB} 
$$
%\\
%&\lim_{m\rightarrow +\infty }\frac{2^{2^{(2m-1)}(p^--\tfrac{1}{2}p^+)}}{2+\tfrac{p^-}{2}2^{2^{(2m-1)}(p^--\tfrac{1}{2}p^+)}}=\tfrac{2}{p^-}>\tfrac{1}{p^-}. % \notag
%\end{align*}
It means that conditions $(F_{4}^\infty),(F_{5}^{\infty})$ are satisfied and by Theorem \ref{tw1'} problem \eqref{ex} with \eqref{fun1} and \eqref{wsp1} admits a sequence of positive solutions in $E$ whose norms tend to infinity. Notice that $(F^\infty_6)$ is not satisfied in this case.
\end{example}
\begin{remark}
In analogously way to Remarks \ref{rem-1} and \ref{rem-2} we can construct an example which show that Theorems \ref{tw1}, \ref{tw1'} are independent and an example of a function which satisfies $(F_1)$, $(F_2)$, $(F_4^\infty)$, $(F_6^\infty)$ and does not satisfy $(F_5^\infty)$.
\end{remark}

\end{document}